\documentclass[10pt,conference]{ieeeconf}

\usepackage{url}            
\usepackage{amsfonts}       
\usepackage{amsmath, amssymb}
\usepackage{graphicx}
\usepackage{algorithm, algorithmicx, algpseudocode,mathtools}
\DeclareMathOperator*{\argmin}{\arg\min}
\DeclareMathOperator*{\argmax}{arg\,max}

\usepackage{soul}
\usepackage{animate}

\usepackage{amssymb}  

\usepackage{amsthm}
\usepackage{bm}
\usepackage{float}
\usepackage{color}
\usepackage{epstopdf}

\newtheorem{Theorem}{Theorem}

\newtheorem{Lemma}{Lemma}

\newtheorem{Problem}{Problem}

\newtheorem{Remark}{Remark}

\newtheorem{Assumption}{Assumption}

\newtheorem{Example}{Example}

\newtheorem{Definition}{Definition}

\makeatletter
\let\NAT@parse\undefined
\makeatother

\usepackage{hyperref}

\makeatletter
\hypersetup{colorlinks=true}
\AtBeginDocument{\@ifpackageloaded{hyperref}
  {\def\@linkcolor{blue}
  \def\@anchorcolor{red}
  \def\@citecolor{red}
  \def\@filecolor{red}
  \def\@urlcolor{red}
  \def\@menucolor{red}
  \def\@pagecolor{red}
\begingroup
  \@makeother\`%
  \@makeother\=%
  \edef\x{%
    \edef\noexpand\x{%
      \endgroup
      \noexpand\toks@{%
        \catcode 96=\noexpand\the\catcode`\noexpand\`\relax
        \catcode 61=\noexpand\the\catcode`\noexpand\=\relax
      }%
    }%
    \noexpand\x
  }%
\x
\@makeother\`
\@makeother\=
}{}}
\makeatother

\IEEEoverridecommandlockouts

\def\BibTeX{{\rm B\kern-.05em{\sc i\kern-.025em b}\kern-.08em
    T\kern-.1667em\lower.7ex\hbox{E}\kern-.125emX}}
    
\begin{document}


\title{\LARGE{\bf Fixed-Time Convergence for a Class of Nonconvex-Nonconcave Min-Max Problems}}

\author{Kunal Garg \and Mayank Baranwal
\thanks{K.~Garg is with the Department of Aerospace Engineering at the University of Michigan, Ann Arbor, MI 48105: \texttt{kgarg@umich.edu}.}
\thanks{M.~Baranwal is with the Division of Data \& Decision Sciences, Tata Consultancy Services Research, Mumbai, 400607 India e-mail: \texttt{baranwal.mayank@tcs.com}.}
}

\maketitle


\begin{abstract}
This study develops a fixed-time convergent saddle point dynamical system for solving min-max problems under a relaxation of standard convexity-concavity assumption. In particular, it is shown that by leveraging the dynamical systems viewpoint of an optimization algorithm, accelerated convergence to a saddle point can be obtained. Instead of requiring the objective function to be strongly-convex--strongly-concave (as necessitated for accelerated convergence of several saddle-point algorithms), uniform fixed-time convergence is guaranteed for functions satisfying only the two-sided Polyak-{\L}ojasiewicz (PL) inequality. A large number of practical problems, including the robust least squares estimation, are known to satisfy the two-sided PL inequality. The proposed method achieves arbitrarily fast convergence compared to any other state-of-the-art method with linear or even super-linear convergence, as also corroborated in numerical case studies.
\end{abstract}

\section{Introduction}
In this paper, we study the problem of solving an optimization problem of the form
\begin{align}
    \min_x\max_y F(x,y),
\end{align}
where $F$ is not necessarily convex in $x$ and/or concave in $y$. As discussed in \cite{razaviyayn2020nonconvex}, these problems appear in several important applications, such as zero-sum games \cite{fiez2020gradient}, network optimization \cite{feijer2010stability} and various domains of machine learning (ML) including adversarial learning \cite{razaviyayn2020nonconvex, goodfellow2014generative} and fair ML \cite{madras2018learning}, to name a few. Most algorithms to solve such min-max problems are designed and analyzed in the discrete-time domain via iterative methods. However, in recent few years, the study of continuous-time optimization (through dynamical systems) methods has emerged as a viable alternative for studying optimization problems, see e.g., \cite{cherukuri2017saddle}. In particular, such min-max problems are solved using saddle-point dynamics (SPD).
The continuous-time perspective of optimization problems provides simple and elegant proof techniques for the convergence of solutions to the equilibrium points using Lyapunov stability theory \cite{garg2018fixed}. 

It is worth noticing that while there is much work on continuous-time optimization, most of it addresses asymptotic or exponential convergence of the solutions to the optimal point, i.e., convergence as time tends to infinity; for an overview, see \cite{cherukuri2017saddle,wibisono2016variational}. Furthermore, 
the \textit{strong} or \textit{strict} convexity-concavity of the objective function is a standard assumption for exponential stability in min-max problems. In \cite{cherukuri2017saddle}, the authors discuss the conditions under which the SPD exhibits global asymptotic convergence. In \cite{qu2019exponential,dhingra2018proximal}, the authors show global exponential stability of the gradient-based method for primal-dual gradient dynamics under a strong convexity-concavity assumption. For convex optimization (minimization) problems, as shown in \cite{karimi2016linear}, the condition can be relaxed by assuming that the objective function satisfies the  \textit{Polyak-{\L}ojasiewicz inequality} (PL inequality), i.e., the objective function is \textit{gradient dominated}. The authors in \cite{yang2020global} (see also \cite{dassampling}) extend the notion of PL inequality for min-max functions by introducing two-sided PL inequality, which is a relaxation of convexity-concavity assumption on the objective function. They study the problem in the discrete-time domain and show the linear rate of convergence (equivalent of exponential convergence in continuous time).

More recently, faster notions of stability of dynamical systems, such as finite-time stability (see \cite{bhat2000finite}) and fixed-time stability (FxTS, see \cite{polyakov2012nonlinear}) have become popular in designing methods of solving optimization problems with an accelerated convergence rate. Our prior efforts in this direction have led to development of fixed-time convergent optimization algorithms for various optimization problems~\cite{garg2018fixed,garg2020fixed,garg2021MVIP}. This study shows that fixed-time convergence of SPD can still be guaranteed without requiring the objective function to be strongly-convex--strongly-concave. In particular, we study the min-max problem under the relaxed two-sided PL inequality and design modified SPD with fixed-time convergence to the saddle point (i.e., the solution of the min-max problem). To the best of the authors' knowledge, this is the first work on a fixed-time stable dynamical systems-based algorithm for solving min-max problems under this relaxed assumption. Moreover, the stronger requirement for Lipschitzness~\cite{yang2020global,dassampling} of the objective function is further relaxed to functions with bounded mixed-derivatives. Thus, the proposed work extends naturally to a larger class of min-max problems, while guaranteeing fastest uniform convergence to the saddle point.

\section{Problem Formulation and Preliminaries}\label{sec:prob_form}
\noindent \textbf{Notation}: The set of reals is denoted by $\mathbb R$. The Euclidean norm of $x\in \mathbb R^n$ is denote by $\|x\|$, and its transpose, by $x^\intercal$. For a given function $F:\mathbb R^n\times\mathbb R^m\rightarrow\mathbb R$, $F^*$ denotes the min-max value of the objective function and $(x^*,y^*)$ denotes the saddle point, i.e., $F(x^*,y^*) = F^*$. The notation $f\in C^k(U,V)$ is used for a function $f:U\rightarrow V$, $U\subseteq \mathbb R^n, V\subseteq \mathbb R^m$ which is $k-$times continuously differentiable. For a multivariate function $F\in C^2(\mathbb R^n\times\mathbb R^m,\mathbb R)$, the partial derivatives are denoted as $\nabla_{x} F(x,y) \triangleq \frac{\partial F}{\partial x}(x,y)$ and $\nabla_{x,y} F(x,y) \triangleq \frac{\partial^2 F}{\partial x \partial y}(x,y)$, where $x\in \mathbb R^{n}, y\in \mathbb R^{m}$. 

In this paper, min-max problems are considered that can be formulated as saddle-point problem for a given $F:\mathbb R^n\times\mathbb R^m\rightarrow\mathbb R$. Formally, this can be stated as 
\begin{equation}\label{min max problem}
    \min_{x\in \mathbb R^n}\max_{y\in \mathbb R^m} F(x,y).
\end{equation}
A point $(x^*,y^*)$ is called as \textit{local saddle-point} of $F$ (as well as local optimal solution of \eqref{min max problem}), if there exist open neighborhoods $U_x\subset \mathbb R^n$ and $U_y\subset \mathbb R^m$ of $x^*$ and $y^*$, respectively, such that for all $(x,y)\in U_x\times U_y$, one has
\begin{equation}
    F(x^*,y) \leq F(x^*,y^*) \leq F(x,y^*).
\end{equation}
The point $(x^*,y^*)$ is \textit{global saddle-point} if $U_x = \mathbb R^n$ and $U_y = \mathbb R^m$. The main problem considered in this paper is as follows.

\begin{Problem}
Given a function $F$ and a user-defined time $0<T<\infty$, design a dynamical system-based algorithm to solve \eqref{min max problem} such that the equilibrium point of the dynamical system is a solution $(x^*, y^*)$ of the problem \eqref{min max problem} and the trajectories of the dynamical system reach $(x^*, y^*)$ within the user-defined time $T$ for each initial condition $(x(0), y(0))\in \mathbb R^n\times\mathbb R^m$.
\end{Problem}

First, we provide an overview of stability theory of dynamical systems that are used later in our analysis. Consider the system
\begin{equation}\label{ex sys}
\dot x = h(x),
\end{equation}
where $x\in \mathbb R^n$, $h: \mathbb R^n \rightarrow \mathbb R^n$ and $h(0)=0$. Assume that the solution to \eqref{ex sys} exists, is unique, and continuous for any $x(0)\in \mathbb R^n$, for all $t\geq 0$. 


\begin{Definition}[\hspace{-0.1pt}\cite{polyakov2012nonlinear}]
The origin is said to be a fixed-time stable equilibrium of \eqref{ex sys} if it is Lyapunov stable and there exists $ T<\infty$ such that $\lim\limits_{t\to T}x(t) = 0$ for each $x(0)\in \mathbb R^n$. 
\end{Definition}

\begin{Lemma}[\hspace{-0.1pt}\cite{polyakov2012nonlinear}]\label{FxTS TH}
Suppose there exists a positive definite function $V\in C^1(\mathcal D,\mathbb R)$, where $\mathcal D\subset\mathbb R^n$ is a neighborhood of the origin, for system \eqref{ex sys} such that 
\begin{equation}\label{dot V eq}
    \dot V(x) \leq -pV(x)^{\alpha}-qV(x)^{\beta}, \; \forall x\in \mathcal D\setminus\{0\},
\end{equation}
with $p,q>0$, $0<\alpha<1$ and $\beta>1$. Then, the origin of \eqref{ex sys} is FxTS with settling time (time of convergence) $T \leq \frac{1}{p(1-\alpha)} + \frac{1}{q(\beta-1)}$.
\end{Lemma}

The local strong or strict convexity-concavity assumption is very commonly used in literature for showing asymptotic convergence of saddle-point dynamics to the optimal solution of \eqref{min max problem} (see, e.g., \cite{cherukuri2017saddle}). Authors in \cite{cherukuri2017saddle} use the following nominal saddle-point dynamics (SPD):
\begin{equation}\label{SP asym dyn}
    \dot x = -\nabla F_x(x,y), \quad
    \dot y  = \nabla F_y(x,y).
\end{equation}
and show asymptotic convergence to the saddle-point $(x^*, y^*)$ under the strong and strict convexity-concavity assumption. 

In the context of minimization problems, the requirement of strong convexity for accelerated optimization can be relaxed to a class of potential non-convex functions satisfying PL inequality~\cite{karimi2016linear}. The notion of \textit{gradient-dominance} or Polyak-{\L}ojasiewicz (PL) inequality has been explored extensively in optimization literature to show exponential convergence. A function $f:\mathbb R^n\rightarrow\mathbb R$ is said to satisfy PL inequality, or is gradient dominated, with $\mu_f>0$ if $\frac{1}{2}\|\nabla f(x)\|^2\geq \mu_f(f(x)-f^*)$ for all $x \in \mathbb R^n$, where $f^* = f(x^*)$ is the value of the function at its minimizer $x^*$. It is easy to show that if a function $f:\mathbb R^m\rightarrow \mathbb R$ is strongly convex, then the function $g:\mathbb R^n\rightarrow \mathbb R$, defined as $g(x) = f(Ax)$, $A\in \mathbb R^{n\times m}$, then $g$ may not be strongly convex if the matrix $A$ is not full row-rank. On the other hand, as shown in \cite[Appendix 2.3]{karimi2016linear}, $g$ still satisfies PL inequality for any matrix $A$. Below, an example of an important class of problems is given for which, the objective function satisfies PL inequality (see \cite{karimi2016linear} for more examples on useful functions that satisfy PL inequality).

\begin{Example}[\textbf{Least squares}]\label{Example: PL inequality} Consider the problem 
\begin{equation}\label{PL example 1}
    \min_{x\in \mathbb R^n} f(Ax) = \|Ax-b\|^2,
\end{equation}
where $x\in \mathbb R^n, A\in \mathbb R^{n\times n}$ and $b\in \mathbb R^n$. Here, the function $f(x) = \|x-b\|^2$ is strongly-convex, and hence, $g(x) = \|Ax-b\|^2$ satisfies PL inequality for any matrix $A$.
\end{Example}

The objective function in \eqref{PL example 1} satisfies PL inequality but need not be strongly convex for any matrix $A$. This is an important class of functions in machine learning problems. Similarly, a notion of PL inequality for min-max functions is introduced in \cite{yang2020global} and is termed two-sided PL inequality. 

\begin{Definition}[\textbf{Two-sided PL inequality}]\label{def: two-sided PL} A continuously differentiable function $F:\mathbb R^n\times\mathbb R^m\rightarrow \mathbb R$ is said to satisfy \textnormal{two-sided PL inequality} if there exist constants $\mu_1, \mu_2>0$ such that for all $x\in \mathbb R^n$ and $y\in \mathbb R^m$,
\begin{subequations}
\begin{align}
\|\nabla_xF(x,y)\|^2 & \geq 2\mu_1 (F(x,y)- \min_xF(x,y)),  \\    
\|\nabla_yF(x,y)\|^2 & \geq 2\mu_2 (\max_y F(x,y)-F(x,y)).
\end{align}
\end{subequations}
\end{Definition}

Below, we give an example of an important class of functions that are not strongly convex-concave but satisfy two-sided PL inequality. 

\begin{Example}
[\textbf{Robust Least Squares}]\label{ex:RLS} Consider the problem
\begin{align}
    \min_x\max_y F(x,y) \coloneqq \|Ax-y\|_M^2 -\lambda \|y-y_0\|_M^2,
\end{align}
where $M\in \mathbb R^{n\times n}$ is a positive semi-definite matrix and  $\|x\|_M\coloneqq \sqrt{x^\intercal Mx}$.  It can be easily shown that $F$ is not strongly convex in $x$ and when $M$ is not full rank, it is not strongly concave in $y$. This is an important class of functions used in formulating least squares under uncertain data \cite{el1997robust}. 
\end{Example}

Per the results in \cite{karimi2016linear}, we propose the result on quadratic-growth of the function $F$ under two-sided PL inequality. 

\begin{Lemma}\label{conj: QB}
    Let $F:\mathbb{R}^n\times\mathbb{R}^m\to\mathbb{R}$ be a function satisfying the two-sided PL inequality with modulii $\mu_1,\mu_2>0$. Then for all $(x,y)\in\mathbb{R}^n\times\mathbb{R}^m$, the following hold
    \begin{subequations}
    \begin{align}
        \|\nabla_xF(x,y)\|&\geq\mu_1\|x-\overline{x}(y)\| \label{eq:Ineq1}\\
        \|\nabla_yF(x,y)\|&\geq\mu_2\|y-\overline{y}(x)\|, \label{eq:Ineq2}
    \end{align}
    \end{subequations}
    where $\overline{x}(y)\coloneqq\argmin_x F(x,y)$, $\overline{y}(x)\coloneqq\argmax_y F(x,y)$.
\end{Lemma}
\noindent The proof is given in Appendix \ref{app: proof Lemma QB}.   

\section{Main results}
In this section, we present the main results on asymptotic stability of the \textit{nominal} SPD \eqref{SP asym dyn} and FxTS of a modified SPD (defined later) to the saddle-point $(x^*, y^*)$. To this end, we make the following assumption on the function $F$. 

\begin{Assumption}\label{asuum: F}
There exist $\mu_1, \mu_2>0$ such that the function $F\in C^{1,1}(\mathbb R^n\times\mathbb R^m,\mathbb R)$ satisfies the two-sided PL inequality with $\mu_1, \mu_2$, and there exists $0\leq c< \min\{\frac{\mu_1}{2}, \frac{\mu_2}{2}\}$ such that the gradients $\nabla_{x} F(x,\cdot)$ and $\nabla_y F(\cdot,y)$ are Lipschitz continuous with constant $c$, for all $(x,y)\in \mathbb R^n\times\mathbb R^m$. 
\end{Assumption}
\begin{Remark}
    Lipschitz continuity of $\nabla_{x} F(x,y)$ w.r.t. $y$, and of $\nabla_y F(x,y)$ w.r.t. $x$ with constant $c$ translates to $\|\nabla_{x,y} F(x,y)\|\leq c$. This is far less restrictive than requiring all first-order derivatives to be Lipschitz continuous, as required in the analysis reported in~\cite{yang2020global,dassampling}. For instance, consider the function $F(x,y)\coloneqq x^2+x^4-y^2-y^4$. Although, the partial derivatives $\nabla_x F(\cdot,y)$ and $\nabla_y F(x,\cdot)$ are not Lipschitz continuous, the function satisfies Assumption~\ref{asuum: F}.
\end{Remark}

We now analyze the stability of the nominal-SPD.
\begin{Theorem}
Consider the SPD \eqref{SP asym dyn} and assume that the function $F$ satisfies Assumption \ref{asuum: F}. Then, the saddle-point $(x^*, y^*)$ is an asymptotically stable equilibrium of \eqref{SP asym dyn}. 
\end{Theorem}
\begin{proof}
Consider the SPD under PL inequality and a candidate Lyapunov function $V:\mathbb R^n\times\mathbb R^m\rightarrow\mathbb R$ defined as $V(x,y) \coloneqq 2(\max_yF(x,y)-\min_xF(x,y))$. Per Definition \ref{def: two-sided PL}, it holds that{\small
\begin{align*}
    V(x,y) \leq &  \frac{1}{\mu_1}\|\nabla_x F(x,y)\|^2 + \frac{1}{\mu_2}\|\nabla_y F(x,y)\|^2\leq  \frac{1}{\mu}\|\nabla F(x,y)\|^2,
\end{align*}}\normalsize
where $\mu = \min\{\mu_1, \mu_2\}$. Thus, it holds that $V(x^*,y^*) = 0$. Also, it can be readily shown that $V(x,y)>0$ for all $(x,y)\neq (x^*,y^*)$, i.e., $V$ is positive definite. Note that $V$ can be re-written as
\begin{align*}
    V\!=\! &\left(\!\max_yF(x,y)\!-\!F(x^*\!,y^*)\!\right)\!+\!\left(\!\max_yF(x,y)\!-\!F(x,y)\!\right)\\
    &\left(F(x^*,y^*)\!-\!\min_xF(x,y)\right)\!+\!\left(F(x,y)\!-\!\min_xF(x,y)\right)\!,
\end{align*} 
Let $\overline{y}(x) = \arg\max_yF(x,y)$ and $\overline{x}(y) = \arg\min_xF(x,y)$. The time derivative of $V$ reads
\begin{align*}
    \dot V = & 2\nabla_xF(x,\overline{y}(x))\dot x-\nabla_xF(x,y)\dot x-\nabla_yF(x,y)\dot z\\
    & -(2\nabla_yF(\overline{x}(y),y)\dot z-\nabla_xF(x,y)\dot x-\nabla_yF(x,y)\dot y)\\
     = & -\|\nabla F(x,y)\|^2 + 2(\nabla_xF(x,\overline{y}(x))-\nabla_xF(x,y))\dot x \\
    & + 2(\nabla_yF(\overline{x}(y),y)-\nabla_yF(x,y))\dot z\\
    \leq & -\|\nabla F(x,y)\|^2 + \frac{2}{\mu_1} c\|\nabla_xF\|\|\overline{y}(x)-y\|\\
    & + \frac{2}{\mu_2} c\|\nabla_yF\|\|\overline{x}(y)-x\|.
\end{align*}
Using Lemma \ref{conj: QB}, we obtain that $\dot V \leq  -\left(1-\frac{2c}{\mu}\right)\|\nabla F(x,y)\|^2$.
Under Assumption \ref{asuum: F}, we have $\|\nabla_{x,y}F(x,y)\|\leq c$ and $\mu>2c$. We obtain that $\dot V(x,y) \leq -\alpha V(x,y)$, where $\alpha = (\mu-2c)>0$. Thus, $(x^*,y^*)$ is asymptotically stable for \eqref{SP asym dyn}. 
\end{proof}

\begin{Remark}
As also noted in \cite{hassan2019proximal} for proximal flows under PL inequality, it is not generally possible to show exponential stability since, in general, the Lyapunov function $V$ is not upper-bounded by a quadratic error $\|x-x^*\|+\|y-y^*\|$. 
\end{Remark}

Inspired from \cite{garg2018fixed}, consider the following FxTS saddle-point dynamics (FxTS SPD)
\begin{equation}\label{SP first order}
\begin{split}
    \begin{bmatrix}\dot x\\\dot y\end{bmatrix} &= -c_1\frac{\tilde \nabla F(x,y)}{\|\nabla F(x,y)\|^\frac{p_1-2}{p_1-1}}-c_2\frac{\tilde \nabla F(x,y)}{\|\nabla F(x,y)\|^\frac{p_2-2}{p_2-1}},\\
\end{split}
\end{equation}
where $c_1, c_2>0$, $p_1>2, 1<p_2<2$, $\tilde \nabla F(x,y) \triangleq \begin{bmatrix}\nabla_x F(x,y)^\intercal & -\nabla_y F(x,y)^\intercal\end{bmatrix}^\intercal$. Note that \eqref{SP asym dyn} is a special case of \eqref{SP first order} with $c_1 = 1, c_2 = 0$ and $p_1 = 2$. The following result can be readily stated for \eqref{SP first order}.

\begin{figure*}[!ht]
	\begin{center}
		\begin{tabular}{ccc}
		\includegraphics[width=0.5\columnwidth]{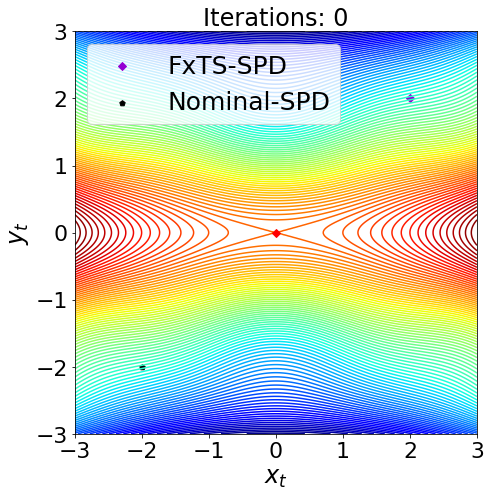}&
			\includegraphics[width=0.5\columnwidth]{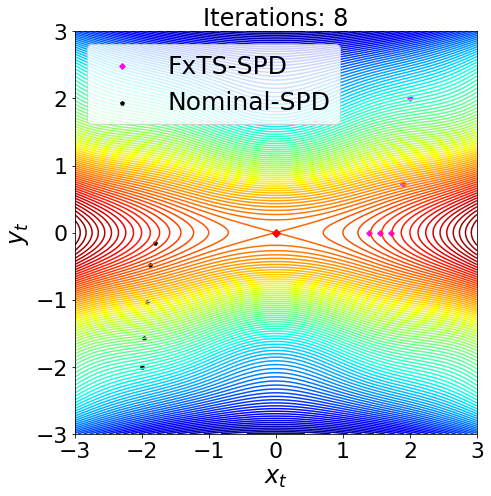}&\includegraphics[width=0.5\columnwidth]{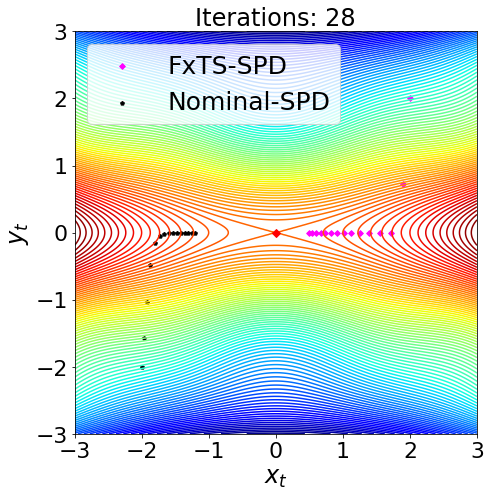}\cr
			\includegraphics[width=0.5\columnwidth]{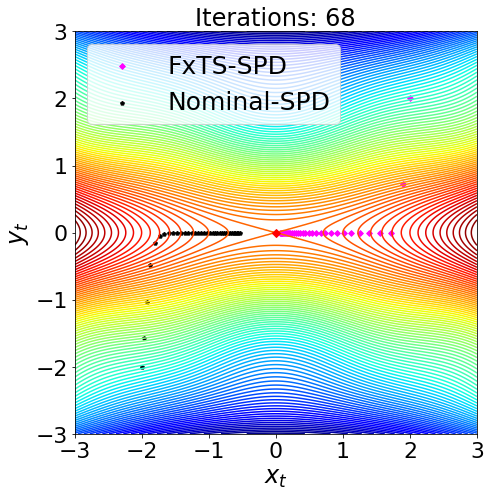}&
			\includegraphics[width=0.5\columnwidth]{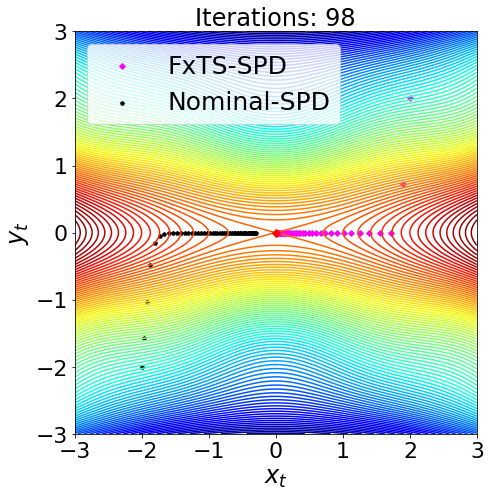}&\includegraphics[width=0.5\columnwidth]{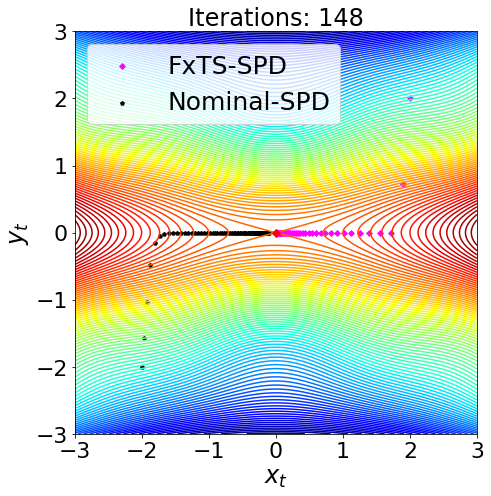}
		\end{tabular}
		\caption{Snapshots of the trajectories of the FxTS-SPD (magenta) amd the nominal-SPD (black) at various iterations.}
		\label{fig:anim}
	\end{center}
	\vspace{-1.5em}
\end{figure*}

\begin{Theorem}\label{thm: FxTS SP PL}
Suppose the function $F$ satisfies Assumption \ref{asuum: F}. Then, the trajectories of \eqref{SP first order} converge to the saddle-point in a fixed time $T<\infty$ for all $(x(0), y(0))\in\mathbb R^n\times\mathbb R^m$. 
\end{Theorem}
\noindent The proof is given in Appendix \ref{app: proof main thm}. 

\begin{Remark}
Note that the difference between the proposed FxTS-SPD \eqref{SP first order} and the nominal-SPD \eqref{SP asym dyn} are the exponents $p_1>2$ and $1<p_2<2$. In the particular (limiting) case of $p_1 = p_2 = 2$, the modified SPD reduces to the nominal-SPD. Intuitively, compared to the asymptotic convergence condition $\dot V\leq -\alpha V$, terms $V^\beta$ and $V^\alpha$ in \eqref{eq: dot V proof} dominate the linear term $V$ when it is large and small, respectively, resulting in accelerated convergence for both small and large initial distance from the equilibrium point. Thus, the FxTS-SPD achieves faster convergence from any initial condition.
\end{Remark}


\section{Numerical case studies}
We now present numerical experiments on robust least square estimation, as well as study convergence behavior for a toy example. While continuous-time algorithms are useful from the point of view of analysis, these algorithms are implemented in an iterative, discrete-time manner. For the sake of implementation, we use forward-Euler discretization for both the FxTS-SPD and the nominal-SPD with timescale separation~\cite{fiez2020gradient}, i.e., the gradient ascent dynamics is discretized at a faster rate than the gradient descent dynamics. The algorithms were implemented using PyTorch 0.4.1 on a 16GB Core-i7 2.8GHz
CPU.

We first present the convergence analysis of the proposed FxTS-SPD on a simple two-dimensional nonconvex-nonconcave function that satisfies the two-sided PL inequality. In particular, we consider evaluating the saddle point of the example function $F(x,y)=x^2+3\sin^2x\sin^2y - 4y^2 -10\sin^2y$. We also compare the convergence behavior of the FxTS-SPD against the nominal-SPD (without gradient normalization). Figure~\ref{fig:anim} represents snapshots of the trajectories of the FxTS-SPD (magenta) amd the nominal-SPD (black) at various iterations. The two trajectories are initialized at (2,2) and (-2,-2), respectively. Recall that both initializations are symmetric for the test function. It can be seen that the FxTS-SPD converges to the unique saddle point $(0,0)$ of $F(x,y)$ by the end of 70 iterations, while the nominal-SPD continues to require further computations and does not converge even by the end of 150 iterations. It can further be observed that the paths traversed by both these optimization algorithms are identical (assuming they start from the same initial conditions). This reaffirms the fact that while the two algorithms traverse the same curve, the proposed FxTS-SPD traverses the curve much faster than its nominal counterpart. Thus, this clever reparameterization of the curve through gradient normalization significantly expedites the optimization process{\footnote{Full video available at: \url{https://tinyurl.com/2wpmyf3h}}}.

\begin{figure*}[!ht]
	\begin{center}
		\begin{tabular}{cc}
			\includegraphics[width=0.7\columnwidth]{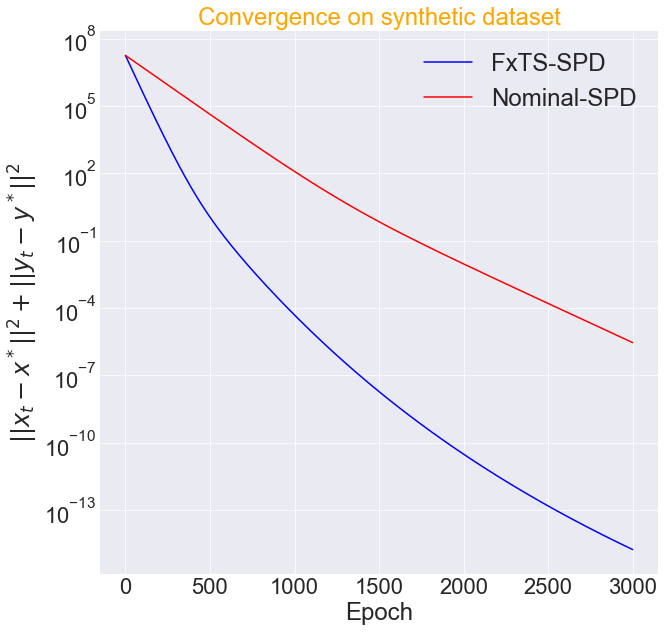}&
			\includegraphics[width=0.7\columnwidth]{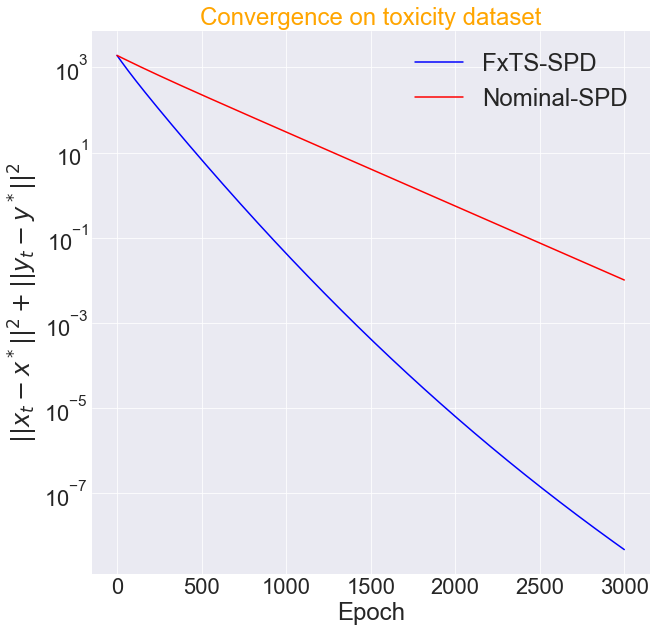}\cr
			(a)&(b)
		\end{tabular}
		\caption{Comparison of the convergences of FxTS-SPD and normal-SPD on two datasets: (a) Synthetic dataset, (b) Toxicity dataset~\cite{cassotti2014prediction}.}
		\label{fig:compare}
	\end{center}
	\vspace{-1.5em}
\end{figure*}

The second set of experiments concern with the problem of robust least square (RLS) estimation~\ref{ex:RLS}. The RLS minimizes the worst case residual given the bounded deterministic perturbation $\delta$ on the noisy measurement vector $y_0\in\mathbb{R}^m$, and the measurement matrix $A\in\mathbb{R}^{m\times n}$. The RLS can be formulated as:
\begin{align*}
    \min\limits_{x\in\mathbb{R}^n}\max\limits_{\delta:\|\delta\|\leq\rho} \|Ax-y\|^2, \qquad \text{where} \quad \delta=y_0-y.
\end{align*}
We consider the soft formulation (see Example~\ref{ex:RLS}). 
Note that for $\lambda>1$, $F(x,y)$ satisfies the two-sided PL inequality, since $F(x,y)$ can be written as a combination of an affine function and a strongly-convex--strongly-concave function. Within RLS estimation, we consider two scenarios: (a) Synthetic dataset: We generate the measurement matrix $A$ with $n=1000$ and $m=500$ by sampling its rows from a normal distribution $\mathcal{N}(0,I_n)$. The noisy measurement $y_0$ is set to $Ax^*+\varepsilon$ for the true signal $x^*$ and an $\varepsilon$ sampled from $\mathcal{N}(0,0.01)$. We use $M=I_n, \lambda=3$ for this dataset. (b) Aquatic toxicity dataset~\cite{cassotti2014prediction}: The dataset was used to predict acute aquatic toxicity of 546 molecules towards Daphnia Magna from 8 descriptors. Here, we set $M=I_n, \lambda=2$.

Figures~\ref{fig:compare}a and \ref{fig:compare}b depict the convergence behaviors of the FxTS-SPD and the nominal-SPD for the two datasets, respectively. We use distance $\|(x_t,y_t)-(x^*,y^*)\|^2$ to the limit point $(x^*,y^*)$ as a metric to compare rates of convergence. In both the scenarios, the convergence rate of the FxTS-SPD is orders of magnitude faster than the nominal-SPD.

\section{Conclusions}
In this paper, we proposed a method of solving a subclass of nonconvex-nonconcave min-max problem within a fixed amount of time. We showed our results under the relaxed assumption of two-sided PL inequality, which is a weaker condition as compared to the commonly used assumptions of strong or strict convexity-concavity. Numerical case studies illustrate significant improvement in convergence performance compared to the nominal-SPD.

\bibliographystyle{IEEEtran}
\bibliography{myreferences}

\appendices

\section{Proof of Lemma \ref{conj: QB}}\label{app: proof Lemma QB}
\begin{proof}
    For the sake of brevity, we only prove the first implication~\eqref{eq:Ineq1}. The proof for the second implication follows similarly. Let us define $h_y(x)\coloneqq\sqrt{F(x,y)-F(\overline{x}(y),y)}$. Clearly, by definition, $h_y(x)\geq 0$ for all $(x,y)\in\mathbb{R}^n\times\mathbb{R}^m$. The partial derivative of $h_y(x)$ for all $x\neq\overline{x}(y)$ is given by:
    \begin{align}\label{eq:h_der}
        \nabla_xh_y(x) &= \frac{\nabla_xF(x,y)}{2\sqrt{F(x,y)-F(\overline{x}(y),y)}} \nonumber \\
        \implies\!\!\|\nabla_xh_y(x)\|^2 &= \frac{\|\nabla_xF(x,y)\|^2}{4\left(F(x,y)-F(\overline{x}(y),y)\right)}\!\geq\! \frac{\mu_1}{2},
    \end{align}
    where the last inequality follows from the two-sided PL inequality. For any point $(x_0,y)$ with $x_0\neq \overline{x}(y)$, consider solving the following differential equation:
    \begin{align}\label{eq:ODE}
        \frac{dx(t)}{dt} = -\nabla_xh_y(x(t)) \ \ \text{with} \ x(t=0)=x_0.
    \end{align}
    Recall that $h_y(x)$ is a positive invex function bounded from below by 0, while the derivative $\nabla_xh_y(x)$, too, is bounded from below by a positive constant. Thus, it must be that by moving along the path defined by~\eqref{eq:ODE}, the trajectory will eventually reach $\overline{x}(y)$ in some finite time $T>0$. Thus,
    \begin{align}\label{eq:h_ineq}
        h_y(x_0) - h_y(x(T)) &= -\int\limits_{x_0}^{x(T)}\langle\nabla_xh_y(x),dx\rangle \nonumber \\
        &= -\int\limits_{0}^{T}\left\langle\nabla_xh_y(x),\frac{dx}{dt}\right\rangle dt \nonumber \\
        &= \int\limits_{0}^{T}\|\nabla_xh_y(x)\|^2dt \geq \frac{\mu_1}{2}T,
    \end{align}
    where the last inequality follows from~\eqref{eq:h_der}. Since, $x(T)=\overline{x}(y)$, i.e., $ h_y(x(T))=0$, it follows that $T\leq 2h_y(x_0)\big/\mu_1$. The \emph{length} of the orbit $x(\cdot)$ defined by~\eqref{eq:ODE} starting at $x_0$ is:
    \begin{align}\label{eq:Lx0}
      \mathcal{L}_y(x_0) \coloneqq \int\limits_{0}^{T}\|dx\big/dt\|dt = \int\limits_{0}^{T}\!\|\nabla_xh_y(x)\|dt \geq \|x_0-\overline{x}(y)\|,
    \end{align}
    where the last inequality follows from the fact that the path-length must be greater than or equal to straight line distance. Revisiting~\eqref{eq:h_ineq}, we obtain:
    \begin{align}\label{eq:h_fin}
        h_y(x_0) - h_y(x(T)) &= \int\limits_{0}^{T}\|\nabla_xh_y(x)\|^2dt \nonumber \\
        &\stackrel{\eqref{eq:h_ineq}}{\geq}\sqrt{\frac{\mu_1}{2}}\int\limits_{0}^{T}\|\nabla_xh_y(x)\|dt \nonumber \\
        &\stackrel{\eqref{eq:Lx0}}{\geq}\sqrt{\frac{\mu_1}{2}}\|x_0-\overline{x}(y)\|.
    \end{align}
    Since $h_y(x(T))=0$,~\eqref{eq:h_fin} implies that $ F(x_0,y)-F(\overline{x}(y),y) \geq \frac{\mu_1}{2}\|x_0-\overline{x}(y)\|^2$, which using the two-sided PL inequality reduces to $\|\nabla_xF(x_0,y)\| \geq \mu_1\|x_0-\overline{x}(y)\|$, implying component-wise quadratic growth.
\end{proof}

\section{Proof of Theorem \ref{thm: FxTS SP PL}}\label{app: proof main thm}
\begin{proof}
Consider the Lyapunov candidate \small{
\begin{align*}
    V(x,y)\!=\! &\left(\!\max_yF(x,y)\!-\!F(x^*\!,y^*)\!\right)\!+\!\left(\!\max_yF(x,y)\!-\!F(x,y)\!\right)\nonumber\\
    &\left(F(x^*,y^*)\!-\!\min_xF(x,y)\right)\!+\!\left(F(x,y)\!-\!\min_xF(x,y)\right)\!,
\end{align*}}\normalsize
and take its time derivative along \eqref{SP first order} to obtain
\small{
\begin{align*}
  \dot V = & 2\nabla_xF(x,\overline{y}(x))\dot x-\nabla_xF(x,y)\dot x-\nabla_yF(x,y)\dot z\\
    & -(2\nabla_yF(\overline{x}(y),y)\dot z-\nabla_xF(x,y)\dot x-\nabla_yF(x,y)\dot y).
\end{align*}}\normalsize
Substituting the dynamics from \eqref{SP first order} and performing routine calculation, we obtain
\begin{align*}
    \dot V = & -c_1\frac{\|\nabla F\|^2}{\|\nabla F(x,y)\|^\frac{p_1-2}{p_1-1}}-c_2\frac{\|\nabla F\|^2}{\|\nabla F(x,y)\|^\frac{p_2-2}{p_2-1}}\\
    & -2c_1\left(\nabla_x F(x,\overline{y}(x))-\nabla_x F(x,y)\right)^\intercal\frac{\nabla_x F(x,y)}{\|\nabla F(x,y)\|^\frac{p_1-2}{p_1-1}}\\
    & -2c_2\left(\nabla_x F(x,\overline{y}(x))-\nabla_x F(x,y)\right)^\intercal\frac{\nabla_x F(x,y)}{\|\nabla F(x,y)\|^\frac{p_2-2}{p_2-1}}\\
    & + 2c_1\left(\nabla_y F(\overline{x}(y),y)-\nabla_y F(x,y)\right)^\intercal\frac{\nabla_y F(x,y)}{\|\nabla F(x,y)\|^\frac{p_1-2}{p_1-1}}\\
    & + 2c_2\left(\nabla_y F(\overline{x}(y),y)-\nabla_y F(x,y)\right)^\intercal\frac{\nabla_y F(x,y)}{\|\nabla F(x,y)\|^\frac{p_2-2}{p_2-1}}.
\end{align*}
Now, under Assumption \ref{asuum: F}, using Lipschitz continuity of the gradients $\nabla_x F$ and $\nabla_y F$, we obtain {\small
\begin{align*}
    \dot V \leq & -c_1\frac{\|\nabla F(x,y)\|^2}{\|\nabla F(x,y)\|^\frac{p_1-2}{p_1-1}}-c_2\frac{\|\nabla F(x,y)\|^2}{\|\nabla F(x,y)\|^\frac{p_2-2}{p_2-1}}\\
    & +2cc_1\frac{\|\nabla_x F(x,y)\|\|x-\overline{x}(y)\|}{\|\nabla F(x,y)\|^\frac{p_1-2}{p_1-1}}\\
    & + 2cc_2\frac{\|\nabla_x F(x,y)\|\|x-\overline{x}(y)\|}{\|\nabla F(x,y)\|^\frac{p_2-2}{p_2-1}}\\
    & + 2cc_1\frac{\|\nabla_y F(x,y)\|\|z-\overline{y}(x)\|}{\|\nabla F(x,y)\|^\frac{p_1-2}{p_1-1}}\\
    & + 2cc_2\frac{\|\nabla_y F(x,y)\|\|z-\overline{y}(x)\|}{\|\nabla F(x,y)\|^\frac{p_2-2}{p_2-1}},
\end{align*}}\normalsize
where $c = \sup_{x,y}\|\nabla_{x,y}F(x,y)\|$. Using Lemma \ref{conj: QB}, we can further upper-bound the RHS of $\dot V$ as \small{
\begin{align*}
    \dot V \leq & -c_1\frac{\|\nabla F(x,y)\|^2}{\|\nabla F(x,y)\|^\frac{p_1-2}{p_1-1}}-c_2\frac{\|\nabla F(x,y)\|^2}{\|\nabla F(x,y)\|^\frac{p_2-2}{p_2-1}}\\
    & +\frac{2cc_1}{\mu_1}\frac{\|\nabla_x F(x,y)\|^2}{\|\nabla F(x,y)\|^\frac{p_1-2}{p_1-1}} +\frac{2cc_2}{\mu_1}\frac{\|\nabla_x F(x,y)\|^2}{\|\nabla F(x,y)\|^\frac{p_2-2}{p_2-1}}\\
    & +\frac{2cc_1}{\mu_2}\frac{\|\nabla_y F(x,y)\|^2}{\|\nabla F(x,y)\|^\frac{p_1-2}{p_1-1}} + \frac{2cc_1}{\mu_2}\frac{\|\nabla_y F(x,y)\|^2}{\|\nabla F(x,y)\|^\frac{p_2-2}{p_2-1}}.
\end{align*}}\normalsize
Thus, it follows that
\small{
\begin{align*}
    \dot V \leq & -c_1\frac{\|\nabla F(x,y)\|^2}{\|\nabla F(x,y)\|^\frac{p_1-2}{p_1-1}}-c_2\frac{\|\nabla F(x,y)\|^2}{\|\nabla F(x,y)\|^\frac{p_2-2}{p_2-1}}\\
    & +\frac{2cc_1}{\mu_1}\frac{\|\nabla F(x,y)\|^2}{\|\nabla F(x,y)\|^\frac{p_1-2}{p_1-1}} +\frac{2cc_2}{\mu_1}\frac{\|\nabla F(x,y)\|^2}{\|\nabla F(x,y)\|^\frac{p_2-2}{p_2-1}} \\
    = & -c_1\left(1-\frac{2c}{\mu_1}\right)\frac{\|\nabla F(x,y)\|^2}{\|\nabla F(x,y)\|^\frac{p_1-2}{p_1-1}}\\
    & -c_2\left(1-\frac{2c}{\mu_2}\right)\frac{\|\nabla F(x,y)\|^2}{\|\nabla F(x,y)\|^\frac{p_2-2}{p_2-1}}.
\end{align*}}\normalsize
Let $\alpha_1 = c_1\left(1-\frac{2c}{\mu_1}\right)$, $\alpha_2 = c_2\left(1-\frac{2c}{\mu_2}\right)$, $\beta_1 = 2-\frac{p_1-2}{p_1-1}$ and $\beta_2 =2- \frac{p_2-2}{p_2-1}$ to obtain
\begin{align*}
    \dot V \leq -\alpha_1 \|\nabla F(x,y)\|^{\beta_1}-\alpha_2 \|\nabla F(x,y)\|^{\beta_2}.
\end{align*}
Per Assumption \ref{asuum: F}, it holds that $\alpha_1, \alpha_2>0$. 
Finally, using two-sided PL property, we obtain that
\begin{align*}
    V(x,y) \leq & \frac{1}{\mu_1}\|\nabla_xF(x,y)\|^2 + \frac{1}{\mu_2}\|\nabla_yF(x,y)\|^2\\
    \leq & \frac{1}{\mu} \|\nabla F(x,y)\|^2,
\end{align*}
where $\mu = \min\{\mu_1,\mu_2\}$. Thus, it holds that
\begin{align}\label{eq: dot V proof}
    \dot V(x,y) \leq -a_1V(x,y)^{b_1}-a_2V(x,y)^{b_2},
\end{align}
where $a_1 = \alpha_1(\mu)^{\frac{\beta_1}{2}}, a_2 = \alpha_2(\mu)^{\frac{\beta_2}{2}}$, $0<b_1 = \frac{\beta_1}{2}<1$ and $b_1 = \frac{\beta_2}{2}>1$. Thus, per Lemma \ref{FxTS TH}, it holds that $V\to 0$ within a fixed time $T<\infty$. Using the positive-definiteness of the function $V$, it holds that the solutions $(x(t), y(t))$ of \eqref{SP first order} converge to the saddle point $(x^*, y^*)$ within time $T$.  
\end{proof}

\end{document}